\documentclass[a4paper,openany, 12pt]{article}

\usepackage{titlesec}
\titleformat*{\section}{\centering\bfseries\Large}
\titleformat*{\subsection}{\centering\bfseries\large}

\usepackage{booktabs}
\usepackage[T1]{fontenc}
\usepackage[utf8]{inputenc}
\usepackage[english]{babel}
\usepackage[a4paper,top=4cm,bottom=3cm,left=2.5cm,right=2.5cm]{geometry}
\usepackage{amsfonts}
\usepackage{mathtools}
\usepackage{amsmath}

\usepackage{lmodern}
\usepackage{empheq}
\usepackage{amsthm}
\usepackage{amssymb}
\usepackage{caption}
\usepackage{siunitx}
\usepackage{float}
\usepackage{amssymb}
\usepackage{graphicx}
\usepackage{stmaryrd}
\usepackage{color}
\usepackage{enumerate}
\usepackage{enumitem}
\setlist[description]{leftmargin=\parindent,labelindent=\parindent}
\usepackage{listings} 
\usepackage{amsmath}
\usepackage{amscd}
\usepackage{braket}
\usepackage{physics}
\usepackage{tikz}
\usepackage{calc}

\newcommand*{\rota}{\mathbb{A}}
\newcommand*{\haus}{\mathcal{H}}
\newcommand*{\sym}{\diamondsuit}

\newcommand*{\F}{\mathcal{F}}

\usepackage[multiple]{footmisc}

\usepackage{ascii}

\numberwithin{equation}{section}

\newtheorem{theorem}{Theorem}[section]
\newtheorem{corollary}[theorem]{Corollary}
\newtheorem{lemma}[theorem]{Lemma}

\theoremstyle{definition}
\newtheorem{definition}[theorem]{Definition}
\newtheorem{example}{Example}[section]

\usepackage{titlesec}
\usepackage{lipsum}

\titleformat{\section}
{\normalfont\scshape}{\thesection}{1em}{}\title{Convergence properties of symmetrization processes}
\author{Jacopo Ulivelli \footnote{\textit{Affiliation}: La Sapienza, University of Rome, Italy, Piazzale Aldo Moro 5, 00185. \textit{Email}: jacopo.ulivelli@gmail.com}}
\date{}
\begin{document}
\maketitle
\begin{abstract}
	Steiner symmetrization is well known for its rounding and general convergence properties. We identify a whole family of symmetrizations sharing analogue behaviors: In fact we prove that all these symmetrizations share the same converging symmetrization processes, together with some pathological phenomena. 

\end{abstract}
	
\section{Introduction}

Symmetrizations play a very important role in geometry and its application, allowing to prove many results with relatively easy and direct proofs. In particular geometric and analytic inequalities, like the Isoperimetric, Blaschke-Santalò, Faber-Krahn inequalities and many others have been proved using for example Steiner symmetrization. For some self contained introductions on the subject see \cite{gardner_2006} Chapter 1 and 2, \cite{GRUB07} Chapter 9, \cite{schneider_2013} Chapter 10, \cite{HUWE} Chapter 3 and the references therein. All these results rely mainly on the fact that through Steiner symmetrization, for every convex compact set it is always possible to find a sequence of symmetrizations converging to a ball while preserving the volume.

We call \textit{symmetrization process} a sequence of symmetrizations applied to a subset of $\mathbb{R}^n$. These processes are the main focus of this work. 
In 1986 Mani-Levitska \cite{MALE1986} showed that, for Steiner symmetrization, a randomly chosen symmetrization process for a convex compact body converges almost surely to a ball. This result was later extended by Van Schaftingen \cite{MR2262256} to general compact sets, then by Volcic \cite{VO2013} for measurable sets. Coupier and Davydov \cite{coupier_davydov_2014} later proved, thanks to the inclusion between Steiner and Minkowski symmetrization, that an analogue probabilistic property holds for Minkowski symmetrization. Other interesting results concerning convergence in probability can be found in the works of Bianchi, Burchard, Gronchi and Volcic \cite{ConvShape} and Burchard and Fortier \cite{BURCHARD2013550}.

Parallelly, the study of deterministic convergence received a boost in 2012 with the work of Klain \cite{KLAIN2012340}, which inspired a series of papers on the subject. In particular Bianchi, Gardner and Gronchi in \cite{BIANCHI201751} and \cite{bianchi2019convergence} introduced a general framework for the study of symmetrizations, focusing, among the others, on the relations between different symmetrizations and their properties. We mainly use their formalism, which now we introduce. Let $\mathcal{E}$ be a family of sets (they might be convex compact sets, epigraphs of a certain class of functions...) and fix a subspace $H$, then a $H$-symmetrization is a map \[\diamondsuit_H: \mathcal{E} \to \mathcal{E}_H, \] where $\mathcal{E}_H$ is the subfamily of $\mathcal{E}$ of $H$-symmetric sets. Here we focus on the family of compact sets $\mathcal{C}^n$ and convex compact sets $\mathcal{K}^n$. Symmetrizations may enjoy many different properties, some of which can completely characterise them, as it is shown for example in \cite{BIANCHI201751}, Section 9. 

The properties we are interested in are monotonicity, invariance under reflections and invariance under orthogonal translations (see next Section for the specific definitions). We call $\F$ the family of all the symmetrizations satisfying such properties for every proper subspace of $\mathbb{R}^n$. The family $\F$ includes Minkowski and fiber symmetrizations (the latter corresponds for convex compact bodies to the Steiner symmetrization when considered with respect to hyperplanes) and possesses a very strong structure, as proved in Corollary 7.3 in \cite{BIANCHI201751}. Such structure plays an essential role in the proof of our main results. Indeed it will allow us to characterise some convergence phenomena for $\F$ which are shared by the whole family. A convergence property that has been studied in literature is the following.

\begin{definition}
	If $\diamondsuit$ is a symmetrization on $\mathcal{K}^n$, a sequence of subspaces $(H_m)$ of $\mathbb{R}^n$ is said to be \textit{weakly-universal} for $\diamondsuit$ if, for every $k \in \mathbb{N}$ we have that the sequence of sets \[K_{m,k}=\diamondsuit_{H_m}\dots \diamondsuit_{H_k} K \] converges for every $K \in \mathcal{K}^n$ with non empty interior to a ball of radius $r(K,k)$, thus we allow such quantity to change with respect to $k$. If $r(K,k)$ is independent from $k$, then $(H_m)$ is said to be \textit{universal} for $\sym$. 
\end{definition}

The definition of universal sequence was introduced in \cite{coupier_davydov_2014} (Theorem 3.1), were the following result was obtained for the family $\mathcal{K}^n_n$ of convex bodies, i.e. convex compact subsets of $\mathbb{R}^n$ with non empty interior.

\begin{theorem}[Coupier and Davydov]\label{T0}
	A sequence of hyperplanes $(H_m)$ in $\mathbb{R}^n$ is universal for Steiner symmetrization in $\mathcal{K}^n_n$ if and only if it is universal for Minkowski symmetrization in $\mathcal{K}^n_n$.
\end{theorem}

Later this Theorem was extended by Bianchi, Gardner and Gronchi together with the introduction of weakly universal sequences. In \cite{BIANCHI201751}, section 8 and \cite{bianchi2019convergence} section 6 and 7 many results were obtained in this direction, in particular we observe that Theorems 7.3 and 7.4 from \cite{bianchi2019convergence} together with Theorem \ref{T0} generalize the latter result to the family $\mathcal{C}^n$.

In all these results the limit of the sequence is a ball. Here instead we study a wider class of sequences, without prescribing specific limit sets.
\begin{definition}\label{D1}
	If $\diamondsuit$ is a symmetrization on $\mathcal{E}$, a sequence of subspaces $(H_m)$ is said to be $\diamondsuit$-\textit{stable} (or \textit{stable for the symmetrization} $\sym$) if, for every $k \in\mathbb{N}$ the sequence defined for $m\geq k$ \[K_m=\diamondsuit_{H_m}\dots\diamondsuit_{H_{k+1}}\diamondsuit_{H_k}K\] converges for every $K \in \mathcal{E}$.
\end{definition}
Notice that the limit might depend on  $K \in \mathcal{K}^n$ and $k \in \mathbb{N}$ both, as we show in the Examples \ref{exa2} and \ref{ex2}.

The  generalization of Theorem \ref{T0} we prove is the following.
\begin{theorem}\label{T1}
	Let $\diamondsuit \in \F$ be a symmetrization on $\mathcal{K}^n$. Then, if $(H_m)$ is a $\diamondsuit$-stable sequence of subspaces of $\mathbb{R}^n$, it is $\spadesuit$-stable for every symmetrization $\spadesuit \in \F$. 
	
In particular this holds for Steiner and Minkowski symmetrizations when $(H_m)$ are hyperplanes.
\end{theorem} 

This turns out to be a simple consequence of a broader result we prove in Theorem \ref{T4}, which involves the so called \textit{convergence in shape}, studied for the first time in \cite{ConvShape}. This kind of convergence, properly defined in the next section, involves a sequence of rotations $(\rota_m)$, which corrects at every step the underlying symmetrization process. Indeed such tool can be used to give convergence to symmetrizations processes which would not naturally have a limit, as the ones we study in Section 4. Using the same notation of Theorem \ref{T1}, we prove in Theorem \ref{T4} that, given two symmetrizations $\sym, \spadesuit \in \F$, we have that $\sym$ is stable in shape if and only if $\spadesuit$ is stable in shape. 

This last result has many interesting consequences. One of them is a partial answer to the question: Does a converging sequence of hyperplanes induce a converging symmetrization process? We find a positive answer when one additional assumption is imposed. 

\begin{theorem}\label{T6}
	Let $(H_m)$ be a sequence of hyperplanes and consider the corresponding normals $(u_m) \subset \mathbb{S}^{n-1}$. Consider moreover $\sym \in \mathcal{F}$ . If the angles $(\alpha_m) \subset [0,\pi/2]$ given by the relation $\abs{u_m \cdot u_{m-1}}=\cos \alpha_m$ are such that \[\sum_{m\in \mathbb{N}} \abs{\alpha_m}<+\infty,\] then the sequence $(H_m)$ converges and it is $\diamondsuit$-stable on $\mathcal{K}^n$.
\end{theorem}

We were not able to prove that such assumption is necessary, but the question is the subject of ongoing research.\\
The structure of this work is the following. In Section 2 we introduce the basic notation and definitions, recalling some tools and instrumental results. In Section 3 we prove Theorem \ref{T4} and investigate some consequences. Finally in Section 4, after presenting a counterexample from \cite{ConvShape}, we prove new ones and show that the same procedures work for all the symmetrizations of the family $\F$.

\section{Preliminaries}

Our ambient space is the family of compact subsets of $\mathbb{R}^n$, where we consider the following operation: Given two subsets $A,B$ of $\mathbb{R}^n$, the Minkowski addition of $A$ and $B$ is the set \[A+B:=\{x+y|x \in A, y \in B\}. \] Such space is a complete metric space with respect to the topology induced by the Hausdorff distance, which for two compact sets $K,L$ is given by \[d_\haus(K,L):=\max\{\inf\{\varepsilon>0, K\subset L+\varepsilon B^n\}, \inf\{\varepsilon>0, L\subset K+\varepsilon B^n\} \}, \] where $B^n$ is the Euclidean $n$-dimensional unitary ball centered in the origin.  See \cite{schneider_2013}, \cite{HUWE}, \cite{HADW57} for the classical theory of convex bodies.

We denote as $\mathcal{C}^n$ and $\mathcal{K}^n$ the families of compact sets and convex compact sets of $\mathbb{R}^n$ respectively. When using the subscripts $\mathcal{C}^n_n, \mathcal{K}^n_n$, we are restricting ourselves to the respective families of bodies, i.e. sets with non-empty interior. Throughout the paper $B(x,r)$ denotes the Euclidean ball centred in $x$ of radius $r>0$, and we write the Euclidean scalar product between two vectors $x,y \in \mathbb{R}^n$ as $x\cdot y$. $\mathbb{S}^{n-1}$ is the unit sphere in $\mathbb{R}^n$.

When dealing with symmetrizations a crucial role is played by some specific set functions on the family of sets we are considering. In particular, for $\mathcal{C}^n$ we are interested in the volume, given by the $n$-dimensional Lebesgue measure of a set which we write as $\lambda_n(\cdot)$. For convex sets we can also consider the mean width, which for $K\in \mathcal{K}^n$ is  \[ \mathcal{W}(K):= \frac{1}{\omega_n}\int_{\mathbb{S}^{n-1}} [h_K(\nu)+h_K(-\nu)] d\mathcal{H}^{n-1}(\nu), \] where $\omega_n$ is the $(n-1)$-dimensional measure of $\mathbb{S}^{n-1}$ and $h_K$ is the support function of $K$, given by \[h_K(x):=\sup\{x\cdot y|y \in K\}, \quad x \in \mathbb{S}^{n-1}.\]

An important property of the Minkowski addition of compact convex sets is that it is equivalent to the sum of support functions in the following sense: 
\begin{equation}\label{e2}
h_K(\cdot)+h_L(\cdot)=h_{K+L}(\cdot)
\end{equation}
for every $K,L \in \mathcal{K}^n$. In this notation, the translation of a set $A$ by a vector $x\in \mathbb{R}^n$ may be written as $A+x$. \\
The symmetric difference of two measurable subsets $A,B$ of $\mathbb{R}^n$ is \[\lambda_n(A\Delta B):=\lambda_n(A\setminus B)+\lambda_n(B \setminus A).\] In $\mathcal{K}^n_n$ this is the Nikodym distance, which is equivalent to the Hausdorff distance. 

We denote as $\mathcal{G}(n,i), 1 \leq i \leq n-1$ the sets of subspaces of $\mathbb{R}^n$ of dimension $1\leq i\leq n-1$. We use the term subspace referring to linear subspaces. For every subset $A$ of $\mathbb{R}^n$ and subspace $H$ we writethe projection of $A$ onto $H$ as $P_H A$. The reflection with respect to a subspace $H$, given by the map 
\begin{equation}\label{e1}
	x \mapsto x-2P_{H^\perp}\{x\},
\end{equation}
is denoted by $R_H$, where $H^\perp$ is the subspace orthogonal to $H$. When a set $A$ of $\mathbb{R}^n$ is such that $A=R_H A$, it is said to be symmetric with respect to $H$, or $H$-symmetric for short. 
 
A fundamental result in the theory of convex bodies is the Brunn-Minkowski inequality. It states that for every compact sets $K,L$ the inequality 
\begin{equation}\label{e3}
	\lambda_n(K+L)^{1/n}\geq \lambda_n(K)^{1/n}+\lambda_n(L)^{1/n}
\end{equation} 
holds, where equality is achieved if and only if $K$ and $L$ are homothetic convex bodies or lower-dimensional convex sets liyng in two parallel affine subspaces. For a complete survey regarding this inequality and its extensions see Gardner \cite{GARDBM}.

We now go back to the concept of $H$-symmetrization introduced in \cite{BIANCHI201751}, \cite{bianchi2019convergence}. Given a family of sets $\mathcal{E}$ and a subspace $H$, an $H$-symmetrization is a map \[\diamondsuit_H:\mathcal{E}\to \mathcal{E}_H \] with $\mathcal{E}_H=\{C \in \mathcal{E} | R_H E=E\}$, where $R_H$ is the map given by \eqref{e1}.\\ Let $\sym$ be a $H$-symmetrization in $\mathcal{E}$. Of particular interest are the following properties:
\begin{enumerate}
	\item (Monotonicity): For every $K,L \in \mathcal{E}$ if $K \subseteq L$, then $\diamondsuit_H K \subseteq \diamondsuit_H L$,
	\item (Idempotence): For every $K \in \mathcal{E}$ holds $\diamondsuit_H K=\diamondsuit_H \diamondsuit_H K$,
	\item ($H^\perp$-translation invariance for $H$-symmetric sets): If $K \in \mathcal{E}$ and $R_H K=K$, then for every $x \in H^\perp$ holds $\diamondsuit_H(K+x)=K$,
	\item (Invariance for $H$-symmetric sets): If $K \in \mathcal{E}$ and $R_H K=K$, $ \Rightarrow\diamondsuit_H K=K$.
	\item ($F$-invariance): There exist a function $F:\mathcal{E}\to \mathbb{R}$ such that $F(K)=F(\diamondsuit_H K)$ for every $K \in \mathcal{E}$.
\end{enumerate}

When we refer to $\sym$ as a symmetrization it is understood that $\sym_H$ is a $H$-symmetrization for every subspace $H$ of $\mathbb{R}^n$. Of particular interest is the following family of symmetrizations:
\[\mathcal{F}:=\{\sym \text{ symmetrization}|\text{ properties 1,3,4 hold }\}.\]
The results obtained in this paper concern mainly Schwarz, Minkowski e fiber Symmetrizations, which we proceed now to present. 

\paragraph{Schwarz Symmetrization} Let $C \in \mathcal{C}^n$, for a fixed $H \in \mathcal{G}(n,i), 1\leq i \leq n-1$. The Schwarz symmetrization of $C$ is the set \[S_H C:=\bigcup_{x \in H}B(x,r_x),\] where $r_x$ is such that $\lambda_{n-i}(K\cap(H^\perp+x))=\lambda_{n-i}(B(x,r_x))$ if $\lambda_{n-i}(K\cap(H^\perp+x))>0$. If the measure of the section a t $x \in H$ is zero but the section is non empty, we replace it with $x$, otherwise we replace such section with the empty set.
From Fubini's Theorem it follows that this symmetrization preserves the volume, thus satisfying property (5) for $F(\cdot)=\lambda_n(\cdot)$. When $i=n-1$ this is better know as \textit{Steiner Symmetrization}, and in general it decreases intrinsic volumes (see \cite{HADW57}, Satz XI, p. 260, or Theorem 10.4.1 in \cite{schneider_2013}). Both in $\mathcal{C}^n$ and $\mathcal{K}^n$ Schwarz symmetrization satisfies properties (1,2,5), while (3,4) hold only for convex sets in the case $i=n-1$.

\paragraph{Minkowski Symmetrization} Let $C \in \mathcal{C}^n$, $H \in \mathcal{G}(n,i), 1\leq i \leq n$. The Minkowski symmetrization of $C$ is the set \[M_H K:=\frac{C+R_H C}{2}.\] Clearly from the definition of $\mathcal{W}(K)$ and \eqref{e2} it preserves the mean width when $C$ is convex, thus property (5) holds for $F(\cdot)=\mathcal{W}(\cdot)$. From the Brunn-Minkowski inequality \eqref{e3} it follows that \[\lambda_n(M_H K)\geq \lambda_n(K).\] 
It may be useful to consider the central Minkowski symmetrization, i.e. \[M_o K=\frac{K-K}{2},\] which is centrally symmetric. If $K$ lies in an affine subspace $H+x, x \in H^\perp$, then we write $M_x K$ for the central Minkowski symmetrization in such affine subspace, i.e. \[M_x K= \frac{K+R_{H^\perp} K}{2}. \]
In $\mathcal{K}^n$ Minkowski symmetrization satisfies all the listed properties, but in $\mathcal{C}^n$ only (1) holds. 

\paragraph{Fiber Symmetrization} Let $C \in \mathcal{C}^n$, $H \in \mathcal{G}(n,i)1\leq i \leq n$. Then its fiber symmetrization is the set \[F_H K=\bigcup_{x \in H} M_x (K\cap(H^\perp +x)).\] This symmetrization can be seen as an hybrid between Schwarz and Minkowski symmetrization, and the underlying operation was introduced by McMullen \cite{MCMU99}.
As Minkowski symmetrization it increases the volume and in $\mathcal{K}^n$ it satisfies property (1,2,3,4), while (2,3,4) fail to hold in $\mathcal{C}^n$. See \cite{ulivelli2021generalization} for some specific examples.

Fiber and Minkowski symmetrizations can be considered the \textit{extremals} of the family $\mathcal{F}$, in a sense that Theorem \ref{T2} will make clear. Moreover other symmetrizations, different from fiber or Minkowski, can be found in $\mathcal{F}$; see for example the remarks after Corollary 7.3 in \cite{BIANCHI201751}, where the following fact is proved.
\begin{theorem}[Bianchi, Gardner and Gronchi]\label{T2}
	Let $H \in \mathcal{G}(n,i), 1 \leq i \leq n-1$, $\mathcal{E}=\mathcal{K}^n$ or $\mathcal{K}_n^n$. If $\diamondsuit \in \mathcal{F}$, then
	\begin{equation}\label{e4}
		F_H K \subseteq \diamondsuit_H K \subseteq M_H K
	\end{equation}
	for every $K \in \mathcal{E}$.
\end{theorem}

In the last section of this paper we study some counterexamples to the convergence of symmetrization processes. We will present one in Example \ref{ex1}, which was independently proved in \cite{BIANCHI2011869} and \cite{BURCHARD2013550}. Such example shows that a dense sequence of directions, if accurately chosen, leads to a non converging symmetrization process. We show in particular that the same construction holds for the whole family of symmetrizations considered in Theorem \ref{T2}. 

For this family of non converging sequences, when dealing with Steiner symmetrization of compact sets, convergence is still possible in a weaker sense, called \textit{convergence in shape}. The next section will be devoted to the study and the generalization of such convergence (see Definition \ref{exa1}). The first result in this direction was achieved in \cite{ConvShape}, Theorem 2.2, with the following result.

\begin{theorem}[Bianchi, Burchard, Gronchi and Volcic]\label{T3}
	Let $(u_m)$ be a sequence in $\mathbb{S}^{n-1}$ such that $u_m\cdot u_{m-1}=\cos \alpha_m$, where $(\alpha_m)$ is a sequence that satisfies $\sum_{m \in \mathbb{N}} \alpha_m^2 < +\infty$. Define moreover $(H_m)$ as the corresponding sequence of orthogonal hyperplanes given by $H_m=u_m^{\perp}$ for every $m \in \mathbb{N}$.
	
	Then there exist a sequence of rotations $R_m$ such that for every non empty compact set $K \subset \mathbb{R}^n$ the sequence \[K_m:=R_m S_{H_m}\dots S_{H_1} K\] converges in Hausdorff distance to a convex compact set L.
\end{theorem}

In Corollary \ref{C1} we obtain a generalization of such result in $\mathcal{K}^n$ for all the symmetrizations in $\mathcal{F}$ with respect to sequences of hyperplanes.

\section{Shape-Stable Symmetrization Processes}

We start noticing some monotonicity properties for volume and mean width with respect to symmetrizations in $\mathcal{F}$.

\begin{lemma} \label{L1}
	Consider $H \in \mathcal{G}(n,i), 1 \leq i \leq n-1,$ and a symmetrization $\diamondsuit_H: \mathcal{K}^n \to (\mathcal{K}^n)_H$ such that $\diamondsuit\in \mathcal{F}$. Then for every $K \in \mathcal{K}^n$ hold \[\lambda_n(\diamondsuit_H K) \geq \lambda_n(K),\quad  \mathcal{W}(K)\geq \mathcal{W}(\diamondsuit_H K).\]
\end{lemma}
\begin{proof}
The first inequality is a consequence of the Brunn-Minkowski inequality \eqref{e3}. Indeed by definition, if $K \in \mathcal{K}^n$ then every $H$-orthogonal section of $F_H K$ is a Minkowski symmetrization of a $H$-orthogonal section of $K$, thus $\lambda_n(F_H K) \geq \lambda_n(K)$. \\
Now, thanks to the inclusions chain \eqref{e4} we have \[\lambda_n(\diamondsuit_H K)\geq \lambda_n(F_H K) \geq \lambda_n(K).\]
	
	For the second inequality, again thanks to \eqref{e4} we have that $\diamondsuit_H K\subseteq M_H K$, thus clearly $h_{\diamondsuit_H K}(u)\leq h_{M_H K}(u)$ for every $u \in \mathbb{S}^{n-1}$ and consequently \[\mathcal{W}(\diamondsuit_H K)\leq \mathcal{W}(M_H K)=\mathcal{W}(K), \] completing the proof.
\end{proof}

We need now a weaker concept of convergence, which generalizes the phenomenon studied in Theorem \ref{T3}.

\begin{definition}[Convergence in Shape]\label{exa1}
	Given a set $C \in \mathcal{E}$, a symmetrization $\sym$ on $\mathcal{E}$ and a sequence of subspaces $(H_m)$, the sequence of symmetrals \[\sym_{H_m}\dots \sym_{H_1} C\] is said to \textit{converge in shape} if there exist a sequence of rotations $(\rota_m)$ such that 
	\begin{equation}\label{e5}
		\rota_m \sym_{H_m}\dots \sym_{H_1} C
	\end{equation}
	converges.
\end{definition}

\begin{definition}[Shape-stable Sequences]
If in Definition \ref{exa1} for every $k \in \mathbb{N}, m \geq k$ the sequence \[\rota_m \sym_{H_m}\dots \sym_{H_k} C\] converges, where $(\rota_m)$ is independent from $C$ and $k$, the sequence of subspaces $(H_m)$ is \textit{shape-stable} in $\mathcal{E}$ for $\sym$. 
\end{definition}

Examples of a shape-stable sequences are presented in Theorem \ref{T3}. Notice that if $\rota_m$ is the identity for every $m \in \mathbb{N}$, then a shape-stable sequence $(H_m)$ is stable.
To better understand these new concepts, let us show some examples.

\begin{example}[Klain's Theorem]\label{exklain}
In \cite{KLAIN2012340} Klain proved that, considered a finite family of hyperplanes $\mathcal{D}=\{G_1,\dots,G_k \}$, a sequence $(H_m)$ such that $H_m \in \mathcal{D}$ for every $m \in \mathbb{N}$ is stable for Steiner symmetrization on $\mathcal{K}^n$. This result is extended in \cite{bianchi2019convergence} to Minkowski, Fiber, Schwarz and more generic symmetrizations, again in the family $\mathcal{K}^n$. In \cite{ulivelli2021generalization} we proved that such result holds as well for the Minkowski symmetrization on compact sets of $\mathbb{R}^n$.
\end{example}

\begin{example}[Stable Sequence]\label{exa2}
	Consider in $\mathbb{R}^2$ the square $Q$ with vertices $(1,0)$,  $(0,-1)$, $(-1,0)$, $(0,1)$ as in Figure \ref{dr3}. Consider the sequence of lines $(H_m)$ where $H_m=span\{(1,0)\}$ when $m$ is even, $H_m=span\{(0,1)\}$ when $m$ is odd, for $m\geq 2$, while $H_1=span\{(\sqrt{2}+1, 1)\}$. Then $(H_m)$ is stable in $\mathcal{K}^2$ for Minkowski symmetrization thanks to Theorem 5.7 \cite{bianchi2019convergence} (see example \ref{exklain}). Now, observe that $M_{H_1} Q$ is the red octagon in the figure, and all the other symmetrizations leave this body unchanged, so that the limit is be exactly $M_{H_1} Q$. If we start from $m\geq 2$ instead the limit is always $Q$.
\end{example}

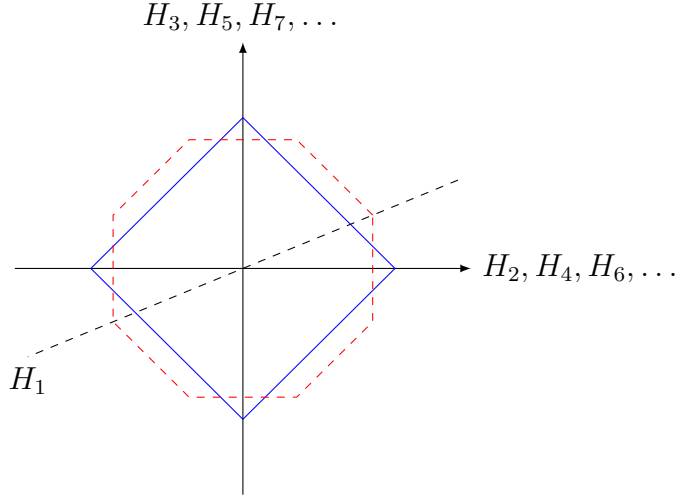
\begin{figure}
	\centering
	\begin{tikzpicture}[scale=1]
    		\draw[-latex] (-3,0) -- (3,0) node[right]{$H_2,H_4,H_6,\dots$};
		\draw[-latex] (0,-3) -- (0,3) node[above]{$H_3,H_5,H_7,\dots$} ;
		
		\draw[blue] (2,0)--(0,2)--(-2,0)--(0,-2)--(2,0); 
		\draw[black, dashed] (2.8284,1.1716)--(-2.8284,-1.1716) node[below]{$H_1$};
		\draw[red,dashed] (1.7071,0.7071)--(0.7071,1.7071)--(-0.7071,1.7071)--(-1.7071,0.7071)--(-1.7071,-0.7071)--(-0.7071,-1.7071)--(0.7071,-1.7071)--(1.7071,-0.7071)--(1.7071,0.7071);
	\end{tikzpicture}
	\caption{Different limits may arise from stable sequences.}
	\label{dr3}
\end{figure}

\begin{example}[Shape-stable Sequence]\label{ex2}
	Consider in $\mathbb{R}^2$ an ellipse $E$ centered in the origin and a sequence of lines as in Theorem \ref{T3}. It is known that Steiner symmetrization preserves ellipses and we can choose a direction $v$ such that the symmetrization with respect to the line $H_1$ parallel to $v$ gives a ball.
			
	If we consider a sequence of lines $(H_m)$ starting from $H_1$ and then continuing as the sequence of Theorem \ref{T3}, we infer that $(H_m)$ is shape-stable in $\mathcal{K}^2$. Moreover, considering that $S_{H_1} E$ is a ball centered in the origin the limit is of course $S_{H_1} E$. If we skip the first symmetrizations, as it was proved in Example 2.1 from \cite{ConvShape} (which we recall here in Example \ref{ex1}) we can choose the remaining directions such that the limit of the convergence in shape is not a ball. 
\end{example}

We can now prove the following equivalence result. Notice that for this result the dimension of the subspaces in the sequence is not relevant.

\begin{theorem}\label{T4}
	Let $\diamondsuit \in \mathcal{F}$. Given a sequence of subspaces $(H_m)$ shape-stable for $\diamondsuit$ in $\mathcal{K}^n$ with isometries $(\rota_m)$, $(H_m)$ is shape-stable in $\mathcal{K}^n$ with isometries $(\rota_m)$ for every symmetrization $\spadesuit \in \mathcal{F}$.
	
	In particular a sequence of subspaces $(H_m)$ is shape-stable for $\diamondsuit \in \mathcal{F}$ if and only if the same property holds for fiber or Minkowski symmetrizations. If $(H_m)$ are hyperplanes, the same conclusion holds for Steiner symmetrization.
\end{theorem}

\begin{proof}[Proof of Theorem \ref{T4}]
	The outline of the proof is the following. We will proceed applying multiple times \eqref{e4}, first proving that if $(H_m)$ is shape stable for $\sym$ in $\mathcal{K}^n$ then it is shape-stable for $M$ in $\mathcal{K}^n$. After that, we show that if the same sequence is shape-stable for $M$  in $\mathcal{K}^n$ then the same holds for $\spadesuit$. 
	
	Let $(H_m)$ be a shape-stable sequence of subspaces for $\sym$ in $\mathcal{K}^n$ as in the hypothesis, we want to prove that for every $K\in \mathcal{K}^n$ the sequence 
\begin{equation}\label{succ}	
K_m=\rota_m M_{H_m}\dots M_{H_1}K
\end{equation}
 converges. Suppose by contradiction that there exists $K \in \mathcal{K}^n$ such that for two subsequences $K_{m_j}, K_{m_l}$ obtained by \eqref{succ} one has $K_{m_j}\to L_1, K_{m_l}\to L_2$ where $L_1\neq L_2$, $L_1,L_2\in \mathcal{K}^n$.
	
Consider the sequence of bodies originated from the same process starting from $K_r:=K+B(0,r)$ instead of $K$, $r>0$ fixed. This is done in order to cover both the full and lower dimensional cases at the same time.\\ Notice that if $H$ is a subspace and $\rota$ is a rotation, for every $K \in \mathcal{K}^n$ \[M_H(K+B(0,r))=M_H(K)+B(0,r), \quad \rota(K+B(0,r))=\rota(K)+B(0,r)\] thus \[(K_r)_m:=\rota_m M_{H_m}\dots M_{H_1}(K+B(0,r))=\rota_m M_{H_m}\dots M_{H_1}K+B(0,r),\] that is $(K_r)_m=K_m+B(0,r)=(K_m)_r$ and instead of $L_1$ and $L_2$ we have the limits $L_1+B(0,r)$ and $L_2+B(0,r)$. Notice that $L_1\neq L_2$ if and only if $L_1+B(0,r)\neq L_2+B(0,r)$ (see for example \cite{schneider_2013}, Lemma 3.1.11).

Since $\lambda_n(K_r)>0$, thanks to Lemma \ref{L1} the sequence of the volumes $\lambda_n(K_m+B(0,r))$ is increasing and strictly positive. Moreover it is bounded, indeed from the compactness of $K$ there exists a ball $B(0,R)$ with $R>0$ such that $K_r \subseteq B(0,R)$. For the monotonicity and symmetry invariance of the Minkowski symmetrization we have that $K_m+B(0,r)\subseteq B(0,R)$ for every $m \in \mathbb{N}$. Then $\lambda_n((K_m)_r)$ converges to a certain value $c_r>0$.

Since $L_1 \neq L_2$, $\lambda_n((L_1)_r\Delta (L_2)_r)=\delta>0$. Fix $0<\varepsilon < \delta/2$, then there exists an index $\nu$ such that $c_r-\lambda_n((K_m)_r)<\varepsilon$ for every $m \geq \nu$. Consider for $m>\nu$ the sequence \[J_m=\rota_m \sym_{H_m}\dots \sym_{H_{\nu+1}}\rota_\nu^{-1}(K_\nu)_r\] \[=\rota_m \diamondsuit_{H_m}\dots \diamondsuit_{H_{\nu+1}} M_{H_\nu}\dots M_{H_1} K_r, \] then thanks to Theorem \ref{T2}, $J_m \subseteq (K_m)_r$ and in particular we have  $J_{m_j}\subseteq (K_{m_j})_r, J_{m_l}\subseteq (K_{m_l})_r$. From the hypothesis the sequence $(H_m)$ is shape-stable in $\mathcal{K}^n$ for $\sym$, thus there exists $J \in \mathcal{K}^n$ such that $J_m\to J$. Clearly the same holds for $(J_{m_j}), (J_{m_l})$. In particular $J\subseteq (L_1)_r$, $J\subseteq (L_2)_r$ and for Lemma \ref{L1} $\lambda_n(J)\geq \lambda_n((K_\nu)_r)$. We infer
	\begin{equation*}
		\begin{gathered}
	\lambda_n((L_1)_r\Delta (L_2)_r)=\lambda_n((L_1)_r\setminus (L_2)_r)+\lambda_n((L_2)_r \setminus (L_1)_r)\leq \\ \lambda_n((L_1)_r\setminus J)+\lambda_n((L_2)_r \setminus J)=2c_r-2\lambda_n(J)\leq 2c_r-2\lambda_n((K_\nu)_r)<2\varepsilon<\delta 
    	\end{gathered}
	\end{equation*}
	which is a contradiction, thus $L_1=L_2$. The same argument can be repeated for every truncated sequence 
	\[\rota_m M_{H_m}\dots M_{H_k} K \]
	 and consequently $(H_m)$ is shape-stable for the Minkowski symmetrization.
	
	Now we prove that if a sequence is shape-stable in $\mathcal{K}^n$ for Minkowski symmetrizations, then it is shape-stable for $\spadesuit \in \mathcal{F}$ as well. Consider for $Z\in \mathcal{K}^n$ the sequence \[Z_m=\rota_m \spadesuit_{H_m} \dots \spadesuit_{H_1} Z.\] Again, by contradiction if $Z_m$ does not converge we can find two different subsequences $Z_{m_j},Z_{m_l}$ converging respectively to $W_1,W_2 \in \mathcal{K}^n$ with $W_1\neq W_2$.
	
	Thanks to Lemma \ref{L1} the sequence $\mathcal{W}(Z_m)$ is non negative and non increasing, thus $\mathcal{W}(Z_m) \to b$ for some $b\geq 0$. Then, if $W_1 \neq W_2$, we have that $\mathcal{W}(\text{conv}(W_1 \cup W_2))>b$. Notice that the cases $W_1 \subset W_2$ and vice versa are automatically excluded since $b=\mathcal{W}(W_1)=\mathcal{W}(W_2)$ and the mean width is strictly monotone. Now, for every $\varepsilon>0$ we can find $\nu \in \mathbb{N}$ such that $\mathcal{W}(Z_m)-b<\varepsilon$ for every $m \geq \nu$. Define for every $m\geq \nu$ the sequence \[V_m=\rota_m M_{H_m}\dots M_{H_{\nu+1}}\rota_\nu^{-1}Z_\nu\] \[=\rota_m M_{H_m}\dots M_{H_{\nu+1}}\spadesuit_{H_\nu}\dots \spadesuit_{H_1}Z \] which converges to some $V \in \mathcal{K}^n$ because $(H_m)$ is shape-stable for Minkowski symmetrization. Minkowski symmetrization preserves the mean width, thus $\mathcal{W}(V)=\mathcal{W}(Z_\nu)$. Moreover, thanks to Theorem \ref{T2} we have that $W_1,W_2 \subseteq V$ and being $V$ convex it holds $\text{conv}(W_1 \cup W_2)\subseteq V$, thus $\mathcal{W}(V)\geq \mathcal{W}(\text{conv}(W_1 \cup W_2))$, then \[\mathcal{W}(\text{conv}(W_1\cup W_2))-b\leq \mathcal{W}(V)-b=\mathcal{W}(Z_\nu)-b<\varepsilon. \] Being $\varepsilon$ arbitrary this inequality contradicts $W(\text{conv}(W_1 \cup W_2))>b$ and $W_1=W_2$. Again the same process can be applied to the truncated sequences, concluding the proof.

\end{proof}

Now Theorem \ref{T1} is just an easy corollary.

\begin{proof}[Proof of Theorem \ref{T1}]
	Observe that if $(H_m)$ is stable then it is shape-stable with $\rota_m$ equal to the identity for every $m$. The proof is then a straightforward application of Theorem \ref{T4}.
\end{proof}

A first consequence is the following extension of Theorem \ref{T0}. Notice that the extension is twofold: The result holds for the whole family $\mathcal{F}$ and the respective family of objects is $\mathcal{K}^n$ instead of $\mathcal{K}^n_n$.
\begin{theorem}\label{cpbis}
Let $\sym,\spadesuit \in \mathcal{F}$. A sequence $(H_m)$ of subspaces is weakly-universal for $\sym$ in $\mathcal{K}^n$ if and only if the same holds for $\spadesuit$.
\end{theorem}
\begin{proof}
The strategy will be the same of Theorem \ref{T4}, with the advantage of using Theorem \ref{T1}.

If $(H_m)$ is weakly-universal for $\sym$ in $\mathcal{K}^n$, then it is stable in $\mathcal{K}^n$. Using Theorem \ref{T1}, this implies that $(H_m)$ is stable for Minkowski symmetrization, thus we only need to prove that for every $K \in \mathcal{K}^n$ the limit $L$ of corresponding symmetrization process \[ K_m=M_{H_m}\dots M_{H_1} K\] is a ball.
Again the sequence $\lambda_n(K_m)$ is bounded and increasing, thus it converges to a certain $c\geq 0$. Through the same argument employed in Theorem \ref{T4} we can suppose $c>0$, i.e. considering $K+B(0,r)$ instead of $K$ for $r>0$ arbitrary small. 

Since $(H_m)$ is weakly-universal for $\sym$, for every $\nu \in \mathbb{N}$ we have that the sequence \[\sym_{H_m}\dots \sym_{H_{\nu+1}} M_{H_\nu}\dots M_{H_1} K \] converges to a ball $B_\nu$ such that $\lambda_n(B_\nu)\geq \lambda_n(K_\nu)$ thanks to Lemma \ref{L1} and $B_\nu \subseteq L$ for every $\nu$ thanks to Theorem \ref{T2}. Since $\lambda_n(K_m) \to c$ increasingly, we have that for every $\varepsilon>0$ exists $\nu \in \mathbb{N}$ such that $\lambda_n(L \Delta B_\nu )<\varepsilon$, thus $L$ is a ball.

Suppose now that $(H_m)$ is weakly universal for Minkowski symmetrization in $\mathcal{K}^n$, then clearly $(H_m)$ is stable for $\spadesuit$. Consider then for $Z \in \mathcal{K}^n$ the limit $W$ of the sequence \[Z_m=\spadesuit_{H_m}\dots \spadesuit_{H_1} Z.\] Again $\mathcal{W}(Z_m)$ is a non negative and non increasing sequence, thus it converges to a value $b\geq0$. 

$(H_m)$ is weakly-universal for Minkowski symmetrization, thus for every $\nu \in \mathbb{N}$ we have a ball $B_\nu$ as the limit of the sequence \[M_{H_m}\dots M_{H_{\nu+1}}\spadesuit_{H_\nu}\dots \spadesuit_{H_1} Z.\] Then $\mathcal{W}(B_\nu)=\mathcal{W}(Z_\nu)$ thanks to the properties of Minkowski symmetrization and for every $\nu$ Theorem \ref{T2} gives $W \subseteq B_\nu$. Since $\mathcal{W}(B_\nu)$ converges decreasingly to $\mathcal{W}(W)$, we have that $W$ must be a ball.
\end{proof}

Theorem \ref{T4} gives us the possibility to extend many known results for Steiner symmetrization to all the $\sym$ satisfying (1,3,4), in particular the Minkowski symmetrization through hyperplanes. For example we immediately have the following generalization of Theorem \ref{T3}.

\begin{corollary}\label{C1}
Let $(H_m)$ be a sequence of hyperplanes and the corresponding normals $(u_m) \subset \mathbb{S}^{n-1}$. Consider moreover $\sym \in \mathcal{F}$ and the angles $(\alpha_m)$ such that $u_m\cdot u_{m-1}=\cos \alpha_m$. If $\sum_{m \in \mathbb{N}} \alpha_m^2 < +\infty$, then there exist rotations $(R_m)$ such that for every non empty convex compact set $K \subset \mathbb{R}^n$ the sequence \[K_m:=R_m \sym_{H_m}\dots \sym_{H_1} K\] converges in Hausdorff distance to a set $L \in \mathcal{K}^n$.
\end{corollary}

For the Minkowski symmetrization with respect to hyperplanes we have a stronger result. First we need the following convergence criterion from \cite{ulivelli2021generalization} (Theorem 1.2).

\begin{theorem}\label{T5}
	Consider $K \in \mathcal{K}^n$ and a sequence of isometries $(\mathbb{A}_m)$. If the sequence \[K_m=\frac{1}{m} \sum_{j=1}^m \mathbb{A}_j K \] converges, then the same happens for every compact set $C \in \mathcal{C}^n$ such that conv$(C)=K$. Moreover, the two sequences converge to the same limit.
\end{theorem}  

Together with Theorem \ref{T1} this implies the following.

\begin{corollary}\label{C2}
	If $(H_m)$ is a shape-stable sequence of hyperplanes for Steiner symmetrization on $\mathcal{C}^n$, then it is shape-stable on $\mathcal{C}^n$ for Minkowski symmetrization. In particular Theorem \ref{T3} holds for Minkowski symmetrization as well.
\end{corollary}
\begin{proof}
	First observe that being $\mathcal{K}^n$ closed in $\mathcal{C}^n$, then $(H_m)$ is obviously shape-stable for Steiner symmetrization on $\mathcal{K}^n$ and from Theorem \ref{T4} it is shape-stable for Minkowski symmetrization on $\mathcal{K}^n$.
	
	Now, in order to conclude the proof we just have to express the shape-stable sequence as a sequence of means of isometries so that we can apply Theorem \ref{T5}. This is clear if we write explicitly the sequence. Indeed, for every $C \in \mathcal{C}^n$ \[C_1=\rota_1 M_{H_1} C=\rota_1 \left(\frac{C+R_{H_1}C_1}{2} \right)=\frac{\rota_1 C+ \rota_1 R_{H_1} C}{2}.\]
Iterating this process every $C_m$ is a Minkowski mean of $2^m$ isometries of $C$.
	
Theorem \ref{T3} provides shape-stable sequences for Steiner symmetrization in $\mathcal{C}^n$, thus the same sequences are shape-stable in $\mathcal{C}^n$ for Minkowski symmetrization.	
\end{proof}

We conclude this section with the proof of Theorem \ref{T6}.
\begin{proof}[Proof of Theorem \ref{T6}]
	In the hypothesis of the theorem we can apply Corollary \ref{C1}. The rotations $R_m$ in such statement (for the details on their construction see the proof of Theorem 2.2 in \cite{ConvShape}, in this paper Theorem \ref{T3}) correspond to the composition of the $m$ planar rotations $\rota_m$ of $\alpha_m$ degrees as $R_m=\rota_m\dots \rota_1$, where every $\rota_m$ is such that $\rota_m R_{m-1} u_{m}=e_1$ and it fixes the subspace $u_m^\perp\cap e_1^\perp$. 
	
	We now show that $(R_m)$ is a Cauchy sequence on the space of endomorphisms of $\mathbb{R}^n$ with the classical sup norm \[\norm{\rota}:=\sup_{z \in \mathbb{S}^{n-1}}\norm{\rota z},\] thus it converges to a certain rotation $R$. From the properties of the norm and the triangular inequality
	\begin{equation*}
		\begin{gathered}
			\norm{R_{m+k}-R_m}=\norm{\rota_{m+k}\dots \rota_{m+1}R_m-R_m}= \norm{R_m}\norm{\rota_{m+k}\dots \rota_{m+1}-Id}=\\ \norm{\rota_{m+k}\dots \rota_{m+1}-Id}\leq \norm{\rota_{m+k}\dots \rota_{m+1}-\rota_{m+1}}+\norm{\rota_{m+1}-Id}\leq\\ \norm{\rota_{m+1}}\norm{\rota_{m+k}\dots \rota_{m+2}-Id}+2\sin(\abs{\alpha_{m+1}}/2)\leq \dots \leq 2\sum_{j=m+1}^{m+k} \sin(\abs{\alpha_j}/2).
		\end{gathered}
	\end{equation*}
	From the hypothesis the series $(\abs{\alpha_m})$ converges, proving the claim.
	
	From Corollary \ref{C1} the sequence of sets \[R_m\sym_{H_m}\dots \sym_{H_1} K\] converges for every $K \in \mathcal{K}^n$ to a certain set $L$, then it makes sense to consider the set $R^{-1}L$. We then have the following estimates \[d_\haus(\sym_{H_m}\dots \sym_{H_1} K,R^{-1}L)\leq d_\haus(\sym_{H_m}\dots \sym_{H_1} K,R_m^{-1}L)+d_\haus(R_m^{-1}L,R^{-1}L).\]
	
	Thanks to the invariance with respect to isometries of the Hausdorff distance \[d_\haus(R_m\sym_{H_m}\dots \sym_{H_1} K,L)=d_\haus(\sym_{H_m}\dots \sym_{H_1} K,R_m^{-1}L)\] and thus from the convergence of $R_m\sym_{H_m}\dots \sym_{H_1} K$ and $R_m$ we infer that $\sym_{H_m}\dots \sym_{H_1} K$ converges to $R^{-1}L$, concluding the proof.
	
\end{proof}

\section{Counterexamples to convergence}

As we have seen Theorem \ref{T4} lets us extend Theorem \ref{T2} to all the symmetrizations $\sym \in \mathcal{F}$. The latter result arises from the study of a peculiar counterexample which now we briefly show. It can be found in different versions in \cite{BIANCHI2011869},\cite{BURCHARD2013550} and \cite{ConvShape}. In the following examples the vectors $\{e_1,e_2\}$ are intended as the standard orthonormal basis in $\mathbb{R}^2$.

\begin{example}\label{ex1}
	Consider a sequence of angles $(\alpha_m) \subset (0,\pi/2)$ such that 
	\begin{equation} \label{e6}
		\sum_{m \in \mathbb{N}} \alpha_m=+\infty, \qquad \sum_{m \in \mathbb{N}} \alpha_m^2< + \infty,
	\end{equation}
	and take the sequence of directions in $\mathbb{R}^2$ given by $u_m:=(\cos \beta_m, \sin \beta_m)$ where \[ \beta_m:=\sum_{j=1}^m \alpha_j\] with corresponding orthogonal lines $H_m=u_m^\perp$.
	
	Let $0<\gamma:=\prod_{m \in \mathbb{N}} \cos \alpha_m$ (which converges because of the second condition in \eqref{e6}), we consider a compact set $K \subset \mathbb{R}^2$ with area $0<|K|<\pi(\gamma/2)^2$ and containing a vertical unitary segment $\ell$ centered in the origin. We prove now that the sequence \[ K_m:=S_{H_m}\dots S_{H_1} K \] does not converge. Consider indeed the sequence of segments \[\ell_m:=P_{H_m} K_{m-1}, \] where the length of $\ell_m$ converges to $\gamma>0$ and each of them is clearly aligned with $H_m$. The sequence of directions $(u_m)$ is dense in $S^1$, thanks to \eqref{e6}, and it does not converge. The same holds for the perpendicular sequence of lines $(H_m)$. Thus for every $\nu \in \mathbb{S}^1$ we can find a subsequence $H_{m_k}^\perp$ such that its corresponding direction converges to $\nu$. Similarly $\ell_{m_k}$ converges to a segment of length $\gamma>0$ aligned to $\nu$. 
	
	Now, if $K_m$ converges, for the monotonicity of Steiner symmetrization it must contain all these subsequences of diameters, and consequently a ball $B$ of diameter $\gamma$ centered  in the origin. But we supposed $\lambda_2(K)<\pi(\gamma/2)^2$, thus $K_m$ cannot converge.
\end{example}

The peculiarity of the sequence involved in this example is that the corresponding directions are dense in $\mathbb{S}^1$, which could seem a reasonable sufficient condition for convergence to a ball. As it was showed this is definitely not the case, even though in \cite{BIANCHI2011869} it was proved for compact convex sets that a dense sequence of hyperplanes can be reordered to obtain a universal sequence. This was generalized in \cite{VO2016} to generic compact sets. 

In \cite{bianchirota} instead it is proved a characterization of the symmetry that a convex compact body needs in order to be a ball. The form we present includes the statements from Theorem 3.2 for one dimensional subspaces.

\begin{theorem}[Bianchi, Gardner and Gronchi]\label{T7}
	Let $H_j \in \mathcal{G}(n,1), j=1,\dots,n$, be such that \\
	(i) at least two of them form an angle that is an irrational multiple of $\pi$,\\
	(ii) $H_1+\dots+H_n=\mathbb{R}^n$ and\\
	(iii)${H_1,\dots,H_n}$ cannot be partitioned into two mutually orthogonal non empty subsets.
	
	If $E\subseteq \mathbb{S}^{n-1}$ is nonempty, closed and such that $R_{H_j}E=E, j=1,...,n$, then $E=\mathbb{S}^{n-1}$.
	
	Hence, if $K \in \mathcal{K}^{n}_n$ satisfies $R_{H_j} K=K$ for $j=1,\dots,n$, then $K$ is a ball centered in the origin.
\end{theorem}

We can use this theorem to find sequences of lines such that the corresponding symmetrization process, if it converges, goes to a ball. Indeed, consider a sequence $(v_m) \subset \mathbb{S}^{n-1}$ with $n$ accumulation points generating a family of lines ${H_1,\dots,H_n}$ as in the statement of Theorem \ref{T7}. Consider indeed a sequence $(K_m)$ of convex bodies such that every $K_m$ is symmetric with respect to $v_m^\perp$. Then, if the sequence converges, the limit must necessarily be a ball. We can use this fact to provide a new kind of counterexample.

\begin{example}
	Consider in $\mathbb{R}^2$ the two directions $w_1=(1,0),w_2=(\cos \alpha, \sin \alpha)$ such that $\alpha>0$ is an irrational multiple of $\pi$. We consider a sequence $(\gamma_m)\subset [0,\alpha]$ such that $(\alpha_m):=(\abs{\gamma_{m+1}-\gamma_m})$ is as in \eqref{e6}. Moreover we want $\alpha$ and $0$ to be accumulation points of $(\gamma_m)$. 

Consider the sequence of lines $H_m=span\{(\cos \gamma_m, \sin \gamma_m)\}$, then the corresponding directions go back and forth between $w_1$ and $w_2$. Observe that $(H_m)$ has two accumulation points aligned with $w_1$ and $w_2$.

Let $K$ be a compact body centered in the origin with a unitary diameter parallel to $w_1$ and consider the sequence of symmetrals \[K_m=S_{H_m}\dots S_{H_1} K.\] As in Example \ref{ex1} we can consider a sequence of segments \[\ell_m=K_m \cap H_m\] such that $\lambda_1(\ell_{m+1})\geq \lambda_1(\ell_m) \cos \alpha_{m+1}$, thus $\lambda_1(\ell_m)$ converges to a certain value $\gamma>0$ and in particular the two limits of the converging subsequences of $(\ell_m)$ respectively aligned with $w_1$ and $w_2$ will have length greater than $\gamma$.

Using Theorem \ref{T7}, if $K_m$ converges then the limit must be a ball. If we choose $\lambda_2(K)< \pi(\gamma/2)^2$, such ball should contain a diameter of length $\gamma$, which is not possible, thus $K_m$ cannot converge.

\end{example}

We conclude proving that Example \ref{ex1} can be generalized for other symmetrizations, again thanks to Theorem \ref{T2}.

\begin{example}
	Consider a set $K \in \mathcal{K}_2^2$ such that it contains a unitary horizontal segment and with mean width $1/2\pi<\mathcal{W}(K)<\gamma$, where $\gamma$ is as in Example \ref{ex1}. In the hypothesis of Theorem \ref{T2}, for $\sym \in \F$ we have that \[S_{U_m}...S_{U_1} K \subseteq \diamondsuit_{U_m}...\diamondsuit_{U_1} K \subseteq M_{U_m}...M_{U_1} K,\] again $U_j:=\text{span}(u_j)$, and we used Steiner symmetrization because it is equivalent to fiber symmetrization relative to a hyperplane, which is our case working on $\mathbb{R}^2$.
	
	In this way we can exploit the first counterexample and the inclusion chain of Theorem \ref{T2} to guarantee that, if a limit exists for $\diamondsuit_{U_m}...\diamondsuit_{U_1} K$ and $M_{U_m}...M_{U_1} K$, proceeding as before it must contain a ball of diameter $\gamma$, therefore this limit must have mean width greater than $\gamma$. In particular this holds for the sequence of Minkowski symmetrals. But Minkowski symmetrization preserves mean width, which we supposed to be less than $\gamma$. This is a contradiction, thus there cannot be a limit.
\end{example}

\section*{Acknowledgements}
The author would like to thank Gabriele Bianchi and Paolo Gronchi for the many discussions and insightful conversations during the preparation of this work.

\bibliographystyle{siam}
\bibliography{biblioconv}
\end{document}